\documentclass[preprint,12pt]{elsarticle}

\usepackage{amssymb}
\usepackage{amsthm}
\usepackage{amsxtra}

\newtheorem{theorem}{Theorem}[section]
\newtheorem{lemma}[theorem]{Lemma}
\newtheorem{corollary}[theorem]{Corollary}
\newtheorem{proposition}[theorem]{Proposition}

\newdefinition{remark}[theorem]{Remark}
\newdefinition{definition}[theorem]{Definition}
\newdefinition{conjecture}[theorem]{Conjecture}
\newdefinition{example}[theorem]{Example}
\newdefinition{notation}[theorem]{Notation}

\begin{document}

\begin{frontmatter}



\title{A proof of the Corrected Beiter conjecture}


\author{Jia Zhao}

\author{Xianke Zhang}

\begin{abstract}
We say that a cyclotomic polynomial $\Phi_{n}(x)$ has order three if
$n$ is the product of three distinct primes, $p<q<r$. Let $A(n)$ be
the largest absolute value of a coefficient of $\Phi_{n}(x)$ and
$M(p)$ be the maximum of $A(pqr)$. In 1968, Sister Marion Beiter
conjectured that $A(pqr)\leqslant \frac{p+1}{2}$. In 2008, Yves
Gallot and Pieter Moree showed that the conjecture is false for
every $p\geqslant 11$, and they proposed the Corrected Beiter
conjecture: $A(pqr)\leqslant \frac{2}{3}p$. Here we will give a
proof of this conjecture.
\end{abstract}


\end{frontmatter}

\section{Introduction}
The $n$th cyclotomic polynomial is the monic polynomial whose roots
are the primitive $n$th roots of unity and are all simple. It is
defined by
$$\Phi_{n}(x)=\prod_{\substack{1\leqslant a\leqslant n\\(a,n)=1}}(x-e^{\frac{2\pi
ia}{n}})=\sum_{i=0}^{\phi(n)}c_{i}x^{i}.$$ The degree of $\Phi_{n}$
is $\phi(n)$, where $\phi$ is the Euler totient function. It is
known that the coefficients $c_{i}$, where $0\leqslant i\leqslant
\phi(n)$, are all integers.
\begin{definition}
$$A(n)=\max\{|c_{i}|,0\leqslant
i\leqslant \phi(n)\}.$$
\end{definition}
For $n<105$, $A(n)=1$. It was once conjectured that this would hold
for all $n$, however $A(105)=2$. Note that 105 is the smallest
positive integer that is the product of three distinct odd primes.
In fact, it is easy to prove that $A(p)=1$ and $A(pq)=1$ for
distinct primes $p,q$. Besides, we have the following useful
propositions.
\begin{proposition}
The nonzero coefficients of $\Phi_{pq}(x)$ alternate between $+1$
and $-1$.
\end{proposition}
\begin{proposition} Let $p$ be a prime.\\
If $p\,|\,n$, then $\Phi_{pn}(x)=\Phi_{n}(x^{p})$, so $A(pn)=A(n)$.\\
If $p\nmid n$, then $\Phi_{pn}(x)=\Phi_{n}(x^{p})/\Phi_{n}(x)$.\\
If $n$ is odd, then $\Phi_{2n}(x)=\Phi_{n}(-x)$, so $A(2n)=A(n)$.
\end{proposition}
\begin{proof}
See \cite{Lenstra} for details.
\end{proof}

By the proposition above, it suffices to consider squarefree values
of $n$ to determine $A(n)$. For squarefree $n$, the number of
distinct odd prime factors of $n$ is the order of the cyclotomic
polynomial $\Phi_{n}$. Therefore the cyclotomic polynomials of order
three are the first non-trivial case with respect to $A(n)$. We also
call them ternary cyclotomic polynomials.

Assume $p<q<r$ are odd primes, Bang \cite{Bang} proved the bound
$A(pqr)\leqslant p-1$. This was improved by Beiter
\cite{Beiter,Beiter2}, who proved that $A(pqr)\leqslant
p-\lfloor\frac{p}{4}\rfloor$, and made the following conjecture.
\begin{conjecture}[(Beiter)]$A(pqr)\leqslant \frac{p+1}{2}$.
\end{conjecture}
Beiter proved her conjecture for $p\leqslant 5$ and also in case
either $q$ or $r\equiv\pm 1\pmod{p}$ \cite{Beiter}. If this
conjecture holds, it is the strongest possible result of this form.
This is because M\"{o}ller \cite{Moller} indicated that for any
prime $p$ there are infinitely many pairs of primes $q<r$ such that
$A(pqr)\geqslant \frac{p+1}{2}$. Define
$$M(p)=\max\{A(pqr)\mid p<q<r\},$$ where the prime $p$ is fixed, and $q$ and
$r$ are arbitrary primes. Now with M\"{o}ller's result, we can
reformulate Beiter's conjecture.
\begin{conjecture}
For $p>2$, we have $M(p)=\frac{p+1}{2}$.
\end{conjecture}
However, Gallot and Moree \cite{Gallot} showed that Beiter's
conjecture is false for every $p\geqslant 11$. Based on extensive
numerical computations, they gave many counter-examples and proposed
the Corrected Beiter conjecture.
\begin{conjecture}[(Corrected Beiter conjecture)]
We have $M(p)\leqslant \frac{2}{3}p$.
\end{conjecture}
This is the strongest corrected version of Beiter's conjecture
because they also proved that for any $\varepsilon >0$,
$\frac{2}{3}p(1-\varepsilon)\leqslant M(p)\leqslant \frac{3}{4}p$
for every sufficiently large prime $p$. In this paper, we will give
a proof of the Corrected Beiter conjecture.

\section{Main theorem}
First we introduce some notation for the rest of the paper. Let
$p<q<r$ be odd primes. Let $$\Phi_{pqr}(x)=\sum_{i}c_{i}x^{i},$$ and
$$\Phi_{pq}(x)=\sum_{m}d_{m}x^{m}.$$ For $i<0$ or
$i>\phi(pqr)=(p-1)(q-1)(r-1)$, $m<0$ or $m>\phi(pq)=(p-1)(q-1)$, we
set $c_{i}=d_{m}=0$.
\begin{notation}
$\forall n\in \mathbb{Z}$, let $\overline{n}$ be the unique integer
such that $0\leqslant \overline{n}\leqslant pq-1$ and
$\overline{n}\equiv n\pmod{pq}$.
\end{notation}
\begin{definition}
For any $n\in \mathbb{Z}$, define a
map$$\chi_{n}:\mathbb{Z}\longrightarrow \{0,\pm 1\}$$ by
$$\chi_{n}(i)=\left \{\begin{array}{ccc} 1 & & \textrm{if}\ \overline{n+p+q}\geqslant
\overline{i+1}>\overline{n+q}\ \textrm{or}\\
   & & \overline{i+1}\leqslant \overline{n+p+q}<\overline{n+q} \ \textrm{or}\\
   & & \overline{n+p+q}<\overline{n+q}<\overline{i+1},\\
-1 & &  \textrm{if}\ \overline{n+p}\geqslant \overline{i+1}>\overline{n} \ \textrm{or}\\
   & &  \overline{i+1}\leqslant \overline{n+p}<\overline{n} \ \textrm{or}\\
   & &  \overline{n+p}<\overline{n}<\overline{i+1},\\
0  & &  otherwise.
\end{array}\right.$$
\end{definition}
It is easy to certify that the value of $\chi_{n}(i)$ only depends
on $\overline{n}$ and $\overline{i}$. That means for any $n',i'\in
\mathbb{Z},n'\equiv n\pmod{pq},i'\equiv i\pmod{pq}$, we have
$$\chi_{n'}(i')=\chi_{n}(i).$$

With notation as above, now we recall some important results. For
the details, we refer the reader to our previous paper \cite{Zhao}.
\begin{lemma}
We have $$c_{i}=\sum_{mr+p+q\geqslant i+1+pq}d_{m}\chi_{mr}(i).$$
\end{lemma}
\begin{corollary}
For any integer $i$, $$\sum_{m}d_{m}\chi_{mr}(i)=0.$$
\end{corollary}
\begin{corollary}
We have $$A(pqr)\leqslant
\max_{i,j\in\mathbb{Z}}\left|\sum_{m\geqslant
j}d_{m}\chi_{mr}(i)\right|.$$
\end{corollary}
\begin{remark}
If $q$ and $r$ interchange, we will have similar arguments as above.
\end{remark}

By corollary 2.3, we know it is sufficient for estimating the upper
bound of $A(pqr)$ to consider
$\max_{i,j\in\mathbb{Z}}\left|\sum_{m\geqslant
j}d_{m}\chi_{mr}(i)\right|.$ Therefore we need to study the
coefficients $d_{m}$ of $\Phi_{pq}.$
\begin{notation}
For any distinct primes $p$ and $q$, let $q_{p}^{*}$ be the unique
integer such that $0<q_{p}^{*}<p$ and $qq_{p}^{*}\equiv 1\pmod{p}$.
Let $\overline{q_{p}}$ be the unique integer such that
$0<\overline{q_{p}}<p$ and $q\equiv \overline{q_{p}}\pmod{p}$.
\end{notation}
About the coefficients of $\Phi_{pq}$, Lam and Leung \cite{Lam}
showed
\begin{theorem}[(T.Y. Lam and K.H. Leung, 1996)]
Let $\Phi_{pq}(x)=\sum_{m}d_{m}x^{m}$. For $0\leqslant m\leqslant
\phi(pq)$, we have

(A) $d_{m}=1$ if and only if $m=up+vq$ for some $u\in [0,
p_{q}^{*}-1]$ and $v\in [0, q_{p}^{*}-1]$;

(B) $d_{m}=-1$ if and only if $m+pq=u'p+v'q$ for some $u'\in
[p_{q}^{*}, q-1]$ and $v'\in [q_{p}^{*}, p-1]$;

(C) $d_{m}=0$ otherwise.

The numbers of terms of the former two kinds are, respectively,
$p_{q}^{*}q_{p}^{*}$ and $(q-p_{q}^{*})(p-q_{p}^{*})$, with
difference $1$ since $(p-1)(q-1)=(p_{q}^{*}-1)p+(q_{p}^{*}-1)q$.
\end{theorem}

About $A(pqr)$, the best known general upper bound to date is due to
Bart{\l}omiej Bzd\c{e}ga \cite{Bzdega}. He gave the following
important result
\begin{theorem}[(Bart{\l}omiej Bzd\c{e}ga, 2008)]
Set $$\alpha=\min\{q_{p}^{*}, r_{p}^{*}, p-q_{p}^{*},
p-r_{p}^{*}\}$$ and $0<\beta<p$ satisfying $\alpha\beta qr\equiv
1\pmod{p}$. Put $\beta^{*}=\min\{\beta, p-\beta\}$. Then we have
$$A(pqr)\leqslant\min\{2\alpha+\beta^{*}, p-\beta^{*}\}.$$
\end{theorem}

Now we can prove our main theorem.
\begin{theorem}[(Corrected Beiter conjecture)]
Assume $p<q<r$ are odd primes. Let
$\Phi_{pqr}(x)=\sum_{i}c_{i}x^{i}$ and
$A(pqr)=\max{\{|c_{i}|,0\leqslant i\leqslant \phi(pqr)\}}$. Then we
have $A(pqr)\leqslant \frac{2}{3}p$.
\end{theorem}
\begin{proof}
Suppose $A(pqr)>\frac{2}{3}p$, we will show that this is a
contradiction. Let us first assume $$1\leqslant p-q_{p}^{*}\leqslant
r_{p}^{*}<p-r_{p}^{*}\leqslant q_{p}^{*}\leqslant p-1.\eqno(2.1)$$
By Remark 1, we can observe that the proof is similar for the other
cases. According to Theorem 2.5, it follows $\alpha=p-q_{p}^{*}$,
$\beta=p-r_{p}^{*}$ and $\beta^{*}=r_{p}^{*}$. Since
$A(pqr)>\frac{2}{3}p$, so we easily get
$$\beta^{*}<\frac{1}{3}p.\eqno(2.2)$$ By Corollary 2.3, we know there
exist a pair of integers $i, j$ such that $$\left|\sum_{m\geqslant
j}d_{m}\chi_{mr}(i)\right|>\frac{2}{3}p.\eqno(2.3)$$ By Theorem 2.4,
we can divide the nonzero terms of $\Phi_{pq}(x)$ into $p$ classes
depending on the value of $v$ or $v'$. From the definition of
$\chi_{n}$, we can simply verify that for any given class, there is
at most one term such that $\chi_{mr}(i)=1$. For the case
$\chi_{mr}=-1$, we have the similar result.

Since $(up+vq)r+p\equiv (up+(v-r_{p}^{*})q)r+p+q\pmod{pq}$, it
follows
$$\chi_{mr}(i)=-1\Longleftrightarrow \chi_{(m-r_{p}^{*}q)r}(i)=1.\eqno(2.4)$$
We claim that $$\sum_{m\geqslant
j}d_{m}\chi_{mr}(i)<-\frac{2}{3}p.\eqno(2.5)$$ By (2.3), we know
$\sum_{m\geqslant j}d_{m}\chi_{mr}(i)>\frac{2}{3}p$ or
$\sum_{m\geqslant j}d_{m}\chi_{mr}(i)<-\frac{2}{3}p$. However the
number of the classes of $d_{m}=-1$ is just $p-q_{p}^{*}$, hence
$$\sum_{m\geqslant j, d_{m}=-1}d_{m}\chi_{mr}(i)\leqslant
\lfloor\frac{1}{3}p\rfloor.$$ If the former holds, then
$$\sum_{m\geqslant j, d_{m}=1}d_{m}\chi_{mr}(i)\geqslant
\lfloor\frac{1}{3}p\rfloor+1\eqno(2.6)$$ Thus there must exist
$u\in[0, p_{q}^{*}-1]$ and $v\in[0,
q_{p}^{*}-1-\lfloor\frac{1}{3}p\rfloor]$ such that
$\chi_{(up+vq)r}(i)=1$. By (2.4), we have
$\chi_{(up+(v+r_{p}^{*})q)r}(i)=-1$ and $v+r_{p}^{*}\in[0,
q_{p}^{*}-1]$. Hence
$d_{up+vq}\chi_{(up+vq)r}(i)+d_{up+(v+r_{p}^{*})q}\chi_{(up+(v+r_{p}^{*})q)r}(i)=0$,
their contributions to the left side of (2.6) are zero. This is a
contradiction, so we establish our claim.

With the arguments above, now we can give the following definition.
For any class $v$, $v\in [0, q_{p}^{*}-1]$, if there exist $u_{1},
u_{2}\in[0, p_{q}^{*}-1]$ such that $u_{1}p+vq\geqslant
j>u_{2}p+vq$, $\chi_{(u_{1}p+vq)r}(i)=-1$ and
$\chi_{(u_{2}p+vq)r}(i)=1$, then we say it is a special class. If
there exists either $u_{1}$ or $u_{2}$ satisfying the above
conditions, we say it is a plain class. If there exists neither
$u_{1}$ nor $u_{2}$ satisfying the above conditions, we say it is a
null class. Similarly for any class $v'$, $v'\in [q_{p}^{*}, p-1]$,
if there exist $u'_{1}, u'_{2}\in[p_{q}^{*}, q-1]$ such that
$u'_{1}p+v'q-pq\geqslant j>u'_{2}p+v'q-pq$,
$\chi_{(u'_{1}p+v'q-pq)r}(i)=1$ and
$\chi_{(u'_{2}p+v'q-pq)r}(i)=-1$, then we say it is a special class.
If there exists either $u'_{1}$ or $u'_{2}$ satisfying the above
conditions, we say it is a plain class. If there exists neither
$u'_{1}$ nor $u'_{2}$ satisfying the above conditions, we say it is
a null class. Let $S$, $P$ and $N$ denote the sets of the special
classes, the plain classes and the null classes respectively.

By Corollary 2.2 and (2.3), we immediately obtain
$$\left|\sum_{m<j}d_{m}\chi_{mr}(i)\right|>\frac{2}{3}p.\eqno(2.7)$$
By (2.3) and (2.7), it is easy to verify that
$$|S|-|N|>\frac{1}{3}p.\eqno(2.8)$$ The number of the classes
$v'$, $v'\in [q_{p}^{*}, p-1]$ is just $p-q_{p}^{*}$, so there must
exist at least one $v$, $v\in [0, q_{p}^{*}-1]$ such that $v\in S$.
Let $v_{0}$ be the largest value of $v\in [0, q_{p}^{*}-1]$ such
that $v\in S$. Next we will consider three cases according to the
value of $v_{0}$ and derive a contradiction to (2.8) to complete the
proof.

\bigskip
\textbf{Case 1.} $q_{p}^{*}-r_{p}^{*}\leqslant v_{0}\leqslant
q_{p}^{*}-1$

First we claim that for any class $v$, $v\in [0, q_{p}^{*}-1]$, if
$v\in S$, then $v_{0}-r_{p}^{*}+1\leqslant v\leqslant v_{0}$.

Obviously we only need to show the first inequality. Suppose
$0\leqslant v\leqslant v_{0}-r_{p}^{*}$ and $v\in S$. By the
definition of the special class, we know there exist $u_{2},
u_{3}\in[0, p_{q}^{*}-1]$ such that $$u_{2}p+v_{0}q<j,
\chi_{(u_{2}p+v_{0}q)r}(i)=1$$ and
$$u_{3}p+vq\geqslant j, \chi_{(u_{3}p+vq)r}(i)=-1.$$ This yields
$$u_{3}p+vq>u_{2}p+v_{0}q,$$ hence $$(u_{3}-u_{2})p>(v_{0}-v)q\geqslant
r_{p}^{*}q.$$ On the other hand, $$(u_{3}-u_{2})p\leqslant
(p_{q}^{*}-1)p=(p-q_{p}^{*})q-p+1\leqslant r_{p}^{*}q-p+1.$$ The
equality holds because $(p-1)(q-1)=(p_{q}^{*}-1)p+(q_{p}^{*}-1)q$.
Therefore we derive a contradiction and prove our claim.

Now we consider the classes $v'$, $v'\in [q_{p}^{*}, p-1]$. Since
$v_{0}\in S$, $v'_{1}=v_{0}+r_{p}^{*}\notin S$. If not, then there
exists $u'_{2}\in[p_{q}^{*}, q-1]$ such that
$\chi_{(u'_{2}p+(v_{0}+r_{p}^{*})q-pq)r}(i)=-1$. By (2.4), we get
$\chi_{((u'_{2}-q)p+v_{0}q)r}(i)=1$. On the other hand, there exists
$u_{2}\in[0, p_{q}^{*}-1]$ such that $\chi_{(u_{2}p+v_{0}q)r}(i)=1$.
By the definition of $\chi_{n}$, we have $q\mid (u'_{2}-u_{2})$,
however it is impossible.

Suppose $q_{p}^{*}\leqslant v'_{2}\leqslant v_{0}+r_{p}^{*}-1$ and
$v'_{2}\in S$. Similarly we know
$v'_{2}-r_{p}^{*}\in[v_{0}-r_{p}^{*}+1, v_{0}]$, but
$v'_{2}-r_{p}^{*}\notin S$. Moreover, if $v'_{2}-r_{p}^{*}\in N$,
then the contributions of these two classes to the left side of
(2.8) are zero, thus we can ignore them. If $v'_{2}-r_{p}^{*}\in P$,
then we say the class $v'_{2}$ is a valid special class. Let $S_{0}$
denote the set of the valid special classes.

Suppose $v_{0}+r_{p}^{*}+1\leqslant v'_{3}\leqslant p-1$ and
$v'_{3}\in S_{0}$. We claim that
$2v_{0}+r_{p}^{*}-v'_{3}\in[v_{0}-r_{p}^{*}+1, v_{0}]$ and
$2v_{0}+r_{p}^{*}-v'_{3}\notin S$.

Since $v_{0}\in S$, there exist $u_{1}, u_{2}\in[0, p_{q}^{*}-1]$
such that $u_{1}p+v_{0}q\geqslant j>u_{2}p+v_{0}q$,
$\chi_{(u_{1}p+v_{0}q)r}(i)=-1$ and $\chi_{(u_{2}p+v_{0}q)r}(i)=1$.
This implies that
$$\overline{(u_{1}p+v_{0}q)r+p+\overline{q_{p}}}=\overline{(u_{2}p+v_{0}q)r+p+q}\eqno(2.9)$$
or
$$\overline{(u_{1}p+v_{0}q)r+p-(p-\overline{q_{p}})}=\overline{(u_{2}p+v_{0}q)r+p+q}.\eqno(2.10)$$
Since $v'_{3}\in S$, there exist $u'_{3}, u'_{4}\in[p_{q}^{*}, q-1]$
such that $u'_{3}p+v'_{3}q-pq\geqslant j>u'_{4}p+v'_{3}q-pq$,
$\chi_{(u'_{3}p+v'_{3}q-pq)r}(i)=1$ and
$\chi_{(u'_{4}p+v'_{3}q-pq)r}(i)=-1$. This implies that
$$\overline{(u'_{3}p+v'_{3}q-pq)r+p+q-\overline{q_{p}}}=\overline{(u'_{4}p+v'_{3}q-pq)r+p}\eqno(2.11)$$
or
$$\overline{(u'_{3}p+v'_{3}q-pq)r+p+q+(p-\overline{q_{p}})}=\overline{(u'_{4}p+v'_{3}q-pq)r+p}.\eqno(2.12)$$
If (2.9) and (2.11) hold simultaneously, then we get
$$\overline{(u_{1}+u'_{3})pr}=\overline{(u_{2}+u'_{4})pr}.$$ Hence
$$q\mid (u_{1}+u'_{3}-u_{2}-u'_{4}).$$ This is impossible. Similarly
(2.10) and (2.12) can not hold simultaneously either, so without
loss of generality we assume (2.10) and (2.11) are correct. By
(2.4), we have $\chi_{(u'_{4}p+v'_{3}q-pq-r_{p}^{*}q)r}(i)=1$.
Because $v'_{3}\in S_{0}$, we know there exists $u_{5}\in[0,
p_{q}^{*}-1]$ such that $u_{5}p+(v'_{3}-r_{p}^{*})q\geqslant j$ and
$\chi_{(u_{5}p+(v'_{3}-r_{p}^{*})q)r}(i)=-1$. Hence
$$\overline{(u_{5}p+(v'_{3}-r_{p}^{*})q)r+p+\overline{q_{p}}}=\overline{((u'_{4}-q)p+(v'_{3}-r_{p}^{*})q)r+p+q}\eqno(2.13)$$
or
$$\overline{(u_{5}p+(v'_{3}-r_{p}^{*})q)r+p-(p-\overline{q_{p}})}=\overline{((u'_{4}-q)p+(v'_{3}-r_{p}^{*})q)r+p+q}.\eqno(2.14)$$
If (2.14) holds, by (2.10) we get
$$\overline{(u_{5}-u_{1})pr}=\overline{(u'_{4}-q-u_{2})pr}.$$ Hence $$q\mid
(u_{5}+u_{2}-u_{1}-u'_{4}).\eqno(2.15)$$ On the other hand, by
$$u_{5}p+(v'_{3}-r_{p}^{*})q\geqslant j>u'_{4}p+v'_{3}q-pq$$ we have
$$0>(u_{5}-u'_{4})p>(r_{p}^{*}-p)q.$$ Note that
$$0>(u_{2}-u_{1})p\geqslant
-(p_{q}^{*}-1)p=-(p-q_{p}^{*})q+p-1\geqslant -r_{p}^{*}q+p-1,$$ so
we can get $$0>(u_{5}+u_{2}-u_{1}-u'_{4})p>-pq+p-1.$$ This
contradicts (2.15) and establishes the validity of (2.13).

Now if $2v_{0}+r_{p}^{*}-v'_{3}\in S$, then there exist $u_{7},
u_{8}\in[0, p_{q}^{*}-1]$ such that
$u_{7}p+(2v_{0}+r_{p}^{*}-v'_{3})q\geqslant
j>u_{8}p+(2v_{0}+r_{p}^{*}-v'_{3})q$,
$\chi_{(u_{7}p+(2v_{0}+r_{p}^{*}-v'_{3})q)r}(i)=-1$ and
$\chi_{(u_{8}p+(2v_{0}+r_{p}^{*}-v'_{3})q)r}(i)=1$. By (2.10), we
have
$$\overline{(u_{7}p+(2v_{0}+r_{p}^{*}-v'_{3})q)r+p-(p-\overline{q_{p}})}
=\overline{(u_{8}p+(2v_{0}+r_{p}^{*}-v'_{3})q)r+p+q}.\eqno(2.16)$$

We recall that $\overline{q_{p}}\leqslant \frac{p-1}{3}$ and refer
the reader to the proof of the main result in \cite{Zhao}. Combining
(2.10), (2.13) and (2.16) yields
\begin{equation*}
\begin{split}
&\overline{((u'_{4}-q)p+(v'_{3}-r_{p}^{*})q)r+p+q}+\overline{(u_{8}p+(2v_{0}+r_{p}^{*}-v'_{3})q)r+p+q}\\
=&\overline{2((u_{2}p+v_{0}q)r+p+q)}
\end{split}
\end{equation*}
Hence $$q\mid u'_{4}+u_{8}-2u_{2}.\eqno(2.17)$$ Moreover, we have
$u_{7}>u_{2}$ because $u_{7}p+(2v_{0}+r_{p}^{*}-v'_{3})q\geqslant
j>u_{2}p+v_{0}q$ and $2v_{0}+r_{p}^{*}-v'_{3}<v_{0}$. It follows
$$u'_{4}-u_{2}>u_{1}-u_{2}=u_{7}-u_{8}>u_{2}-u_{8}.$$
The equality is easily obtained by (2.10) and (2.16). Therefore
$$u'_{4}+u_{8}-2u_{2}=q,$$ and $$2(q-u'_{4})=2(u_{8}-2u_{2})\leqslant 2(p_{q}^{*}-1).\eqno(2.18)$$
On the other hand, by (2.11) and (2.13) we get $$q\mid
u'_{3}+u_{5}-2u'_{4}.$$ In view of the ranges of $u'_{3}$, $u'_{4}$
and $u_{5}$, we certainly have $$u'_{3}+u_{5}-2u'_{4}=0,$$ and
$$2(q-u'_{4})>q+u'_{3}-2u'_{4}=q-u_{5}\geqslant q-p_{q}^{*}+1.\eqno(2.19)$$
By (2.18) and (2.19), we get $$p_{q}^{*}-1>\frac{1}{3}q.$$ Since
$q_{p}^{*}>\frac{2}{3}p$, it is a contradiction, so we prove our
claim.

Finally, we need to show that if
$2(v_{0}+r_{p}^{*})-v'_{3}\in[q_{p}^{*}, v_{0}+r_{p}^{*}-1]$, then
$v'_{3}\in S_{0}$ and $2(v_{0}+r_{p}^{*})-v'_{3}\in S_{0}$ can not
hold simultaneously. Otherwise, there exist $u'_{5},
u'_{6}\in[p_{q}^{*}, q-1]$ such that
$u'_{5}p+(2(v_{0}+r_{p}^{*})-v'_{3})q-pq\geqslant
j>u'_{6}p+(2(v_{0}+r_{p}^{*})-v'_{3})q-pq$,
$\chi_{(u'_{5}p+(2(v_{0}+r_{p}^{*})-v'_{3})q-pq)r}(i)=1$ and
$\chi_{(u'_{6}p+(2(v_{0}+r_{p}^{*})-v'_{3})q-pq)r}(i)=-1$. By (2.4),
we have $$\chi_{(u'_{6}p+(2v_{0}+r_{p}^{*}-v'_{3})q-pq)r}(i)=1.$$
Because $2(v_{0}+r_{p}^{*})-v'_{3}\in S_{0}$, we know there exists
$u_{9}\in[0, p_{q}^{*}-1]$ such that
$u_{9}p+(2v_{0}+r_{p}^{*}-v'_{3})q\geqslant j$ and
$$\chi_{(u_{9}p+(2v_{0}+r_{p}^{*}-v'_{3})q)r}(i)=-1.$$ Hence it
follows
$$\overline{(u_{9}p+(2v_{0}+r_{p}^{*}-v'_{3})q)r+p+\overline{q_{p}}}=\overline{((u'_{6}-q)p+(2v_{0}+r_{p}^{*}-v'_{3})q)r+p+q}.\eqno(2.20)$$
Combining (2.10), (2.13) and (2.20) yields
\begin{equation*}
\begin{split}
&\overline{((u'_{4}-q)p+(v'_{3}-r_{p}^{*})q)r+p+q}+\overline{(u_{9}p+(2v_{0}+r_{p}^{*}-v'_{3})q)r+p}\\
=&\overline{(u_{1}p+v_{0}q)r+p}+\overline{(u_{2}p+v_{0}q)r+p+q}
\end{split}
\end{equation*}
Hence $$q\mid u'_{4}+u_{9}-u_{1}-u_{2}.\eqno(2.21)$$ Moreover
$u_{9}>u_{2}$ and $u'_{4}>u_{1}$, so $$u'_{4}+u_{9}-u_{1}-u_{2}=q,$$
and $$2(q-u'_{4})=2(u_{9}-u_{1}-u_{2})\leqslant
2(p_{q}^{*}-1).\eqno(2.22)$$ By (2.19) and (2.22), we get a
contradiction and establish the claim. Now combining our arguments
above shows that $$|S|-|N|\leqslant r_{p}^{*}<\frac{1}{3}p.$$ This
contradicts (2.8) and completes the proof. In the remaining two
cases, the methods which will be used are similar to the present
case, so we will introduce our ideas and omit the straightforward
details.

\bigskip
\textbf{Case 2.} $r_{p}^{*}\leqslant v_{0}\leqslant
q_{p}^{*}-r_{p}^{*}-1$

First we know for any class $v$, $v\in [0, q_{p}^{*}-1]$, if $v\in
S$, then $v_{0}-r_{p}^{*}+1\leqslant v\leqslant v_{0}$.

Suppose $v_{0}-r_{p}^{*}+1\leqslant v_{1}\leqslant
2v_{0}+r_{p}^{*}-q_{p}^{*}$ and $v_{1}\in S$. Then
$2v_{0}-v_{1}+r_{p}^{*}\geqslant q_{p}^{*}$ and
$2v_{0}-v_{1}+r_{p}^{*}\notin S_{0}$.

For any class $v_{2}$, $2v_{0}+r_{p}^{*}-q_{p}^{*}+1\leqslant
v_{2}\leqslant v_{0}-1$, we consider the class $2v_{0}-v_{2}$.
Obviously, $2v_{0}-v_{2}\notin S$. If $2v_{0}-v_{2}\in P$, then the
class $2v_{0}-v_{2}+r_{p}^{*}\in N$ or $2v_{0}-v_{2}-r_{p}^{*}\in
N$. Combining the arguments above, we get
$$|S|-|N|\leqslant \max{\{r_{p}^{*}, p-q_{p}^{*}+1\}}.$$
By (2.8), it follows $|S|-|N|=p-q_{p}^{*}+1$, so
$$p-q_{p}^{*}=r_{p}^{*}=\frac{p-1}{3}, v_{0}=r_{p}^{*}.\eqno(2.23)$$
If $1\leqslant v_{3}\leqslant v_{0}-1$ and $v_{3}\in S$, then
$v_{3}-r_{p}^{*}+p\in[q_{p}^{*}+1, p-1]$ and
$v_{3}-r_{p}^{*}+p\notin S$. Thus we have
$$|S|=p-q_{p}^{*}+1, |N|=0.\eqno(2.24)$$ We claim that for any
class $v_{4}$, $v_{0}\leqslant v_{4}\leqslant q_{p}^{*}-1$, there
must exist $u_{1}\in[0, p_{q}^{*}-1]$ such that
$u_{1}p+v_{4}q\geqslant j$ and $\chi_{(u_{1}p+v_{4}q)r}(i)=-1$.
Obviously it holds for the classes $v_{0}$ and $q_{p}^{*}-1$. If it
is not correct for a class $v_{4}\in[v_{0}+1, q_{p}^{*}-2]$, then
for the class $v_{4}-r_{p}^{*}$, by (2.4), there does not exist
$u_{4}\in[0, p_{q}^{*}-1]$ such that $u_{4}p+(v_{4}-r_{p}^{*})q<j$
and $\chi_{(u_{4}p+(v_{4}-r_{p}^{*})q)r}(i)=1$. Therefore
$v_{4}-r_{p}^{*}\notin S$, $v_{4}-r_{p}^{*}\in P$. Then there exists
$u_{3}\in[0, p_{q}^{*}-1]$ such that
$u_{3}p+(v_{4}-r_{p}^{*})q\geqslant j$ and
$\chi_{(u_{3}p+(v_{4}-r_{p}^{*})q)r}(i)=-1$. It is not difficult to
obtain that $v_{4}-2r_{p}^{*}+p\notin S$. This implies $|S|\leqslant
p-q_{p}^{*}$. It contradicts (2.24) and establishes the claim.

By (2.24), we know $q_{p}^{*}\in S_{0}$. Similar to the arguments in
case 1, we can get $2v_{0}+r_{p}^{*}-q_{p}^{*}=r_{p}^{*}-1\notin S$.
Thus $p-1\in S$. Moreover, by the previous claim, we know $p-1\in
S_{0}$. In fact, for any $v'_{1}\in[q_{p}^{*}+1, p-2]$, we have
$v'_{1}\in S_{0}$. Otherwise, $v'_{1}\in P$, then
$v'_{1}+r_{p}^{*}-p\in S$ and $2v_{0}+r_{p}^{*}-v'_{1}\notin S$.
This implies the class $2v_{0}-v'_{1}+p\in S$. However,
$v'_{1}+r_{p}^{*}-p\in S$ and $2v_{0}-v'_{1}+p\in S$ can not hold
simultaneously.

With the arguments above, we know there exist $u'_{1},
u'_{2}\in[p_{q}^{*}, q-1]$ such that
$\chi_{(u'_{1}p+(p-1)q-pq)r}(i)=1$ and
$\chi_{(u'_{2}p+(p-1)q-pq)r}(i)=-1$. By (2.11), we get
$$\overline{(u'_{1}p+(p-1)q-pq)r+p+q-\overline{q_{p}}}=\overline{(u'_{2}p+(p-1)q-pq)r+p}.\eqno(2.25)$$
By (2.4), we have $\chi_{((u'_{2}-q)p+2r_{p}^{*}q)r}(i)=1$ and
$\chi_{(u'_{1}p+(r_{p}^{*}-1)q)r}(i)=-1$. By the claim above, we
know there exist $u_{5}, u_{8}\in[0, p_{q}^{*}-1]$ such that
$\chi_{(u_{5}p+2r_{p}^{*}q)r}(i)=-1$ and
$\chi_{(u_{8}p+(r_{p}^{*}-1)q)r}(i)=1$. By (2.13), we get
$$\overline{(u_{5}p+2r_{p}^{*}q)r+p+\overline{q_{p}}}=\overline{((u'_{2}-q)p+2r_{p}^{*}q)r+p+q},\eqno(2.26)$$
and
$$\overline{(u'_{1}p+(r_{p}^{*}-1)q)r+p+\overline{q_{p}}}=\overline{(u_{8}p+(r_{p}^{*}-1)q)r+p+q}.\eqno(2.27)$$
Combining (2.25), (2.26) and (2.27) yields
$$\overline{(u_{5}p+2r_{p}^{*}q)r+p+3\overline{q_{p}}}=\overline{(u_{8}p+(r_{p}^{*}-1)q)r+p+q}.\eqno(2.28)$$
We also have $\chi_{(u_{8}p+(2r_{p}^{*}-1)q)r}(i)=-1$. It follows
$$\overline{(u_{8}p+(r_{p}^{*}-1)q)r+p+q}=\overline{(u_{8}p+(2r_{p}^{*}-1)q)r+p},$$ hence we
have
$$\overline{(u_{5}p+2r_{p}^{*}q)r+p+3\overline{q_{p}}}=\overline{(u_{8}p+(2r_{p}^{*}-1)q)r+p}.\eqno(2.29)$$
Since $v_{0}=r_{p}^{*}\in S$, there exists $u_{9}\in[0,
p_{q}^{*}-1]$ such that $\chi_{(u_{9}p+r_{p}^{*}q)r}(i)=-1$. Thus by
(2.29), we have
$$\overline{(u_{9}p+r_{p}^{*}q)r+p}=\overline{(u_{5}p+2r_{p}^{*}q)r+p+3r_{p}^{*}\overline{q_{p}}}.\eqno(2.30)$$
On the other hand, $\chi_{(u_{5}p+r_{p}^{*}q)r}(i)=1$. By (2.10), we
have
$$\overline{(u_{9}p+r_{p}^{*}q)r+p-(p-\overline{q_{p}})}=\overline{(u_{5}p+r_{p}^{*}q)r+p+q}.\eqno(2.31)$$
(2.30) and (2.31) yields
$$\overline{(u_{9}p+r_{p}^{*}q)r+p}=\overline{(u_{9}p+r_{p}^{*}q)r+p-(p-\overline{q_{p}})+3r_{p}^{*}\overline{q_{p}}}\eqno(2.32)$$
Note that $p=3r_{p}^{*}+1$. Hence (2.32) means that
$\overline{q_{p}}=1$, but it is impossible. Therefore we get a
contradiction and complete the proof of the present case.

\bigskip
\textbf{Case 3.} $0\leqslant v_{0}\leqslant r_{p}^{*}-1$

Suppose $v\in[0, v_{0}]$ and $v\in S$, then we have
$v-r_{p}^{*}+p\leqslant p-1$. If $v-r_{p}^{*}+p\geqslant q_{p}^{*}$,
then $v-r_{p}^{*}+p\notin S$. Therefore we have
$$|S|\leqslant r_{p}^{*}<\frac{1}{3}p.$$ This contradicts
(2.8) and completes the proof of the theorem.
\end{proof}

\bigskip
Department of Mathematical Sciences, Tsinghua University,

Beijing 100084, China

E-mail: zhaojia@mails.tsinghua.edu.cn; xzhang@math.tsinghua.edu.cn

\begin{thebibliography}{100}

\bibitem[1]{Bachman} G. Bachman, On the coefficients of ternary
cyclotomic polynomials, J. Number Theory 100 (2003) 104-116.

\bibitem[2]{Bang} A.S. Bang, Om Lingingen $\Phi_{n}(x)=0$,
Tidsskr. Math. 6 (1895) 6--12.

\bibitem[3]{Beiter} M. Beiter, Magnitude of the coefficients of the
cyclotomic polynomial $F_{pqr}$, Amer. Math. Monthly 75 (1968)
370--372.

\bibitem[4]{Beiter2} M. Beiter, Magnitude of the coefficients of the
cyclotomic polynomial $F_{pqr}$,
\uppercase\expandafter{\romannumeral2}, Duke Math. J. 38 (1971)
591--594.

\bibitem[5]{Bzdega} B. Bzd\c{e}ga, Bounds on ternary cyclotomic
coefficients, arXiv:0812.4024, preprint.

\bibitem[6]{Gallot} Y. Gallot and P. Moree, Ternary cyclotomic polynomials having a large
coefficient, J. Reine Angew. Math. 632(2009), 105-125.

\bibitem[7]{Kaplan} N. Kaplan, Flat cyclotomic polynomials of order
three, J. Number Theory 127 (2007) 118--126.

\bibitem[8]{Lam} T.Y. Lam, K.H. Leung, On the
cyclotomic polynomial $\Phi_{pq}(X)$, Amer. Math. Monthly 103 (1996)
562-564.

\bibitem[9]{Lenstra} H.W. Lenstra, Vanishing sums of roots of
unity, in:Proceedings, Bicentennial Congress Wiskundig Genootschap,
Vrije Univ., Amsterdam, 1978,
Part\uppercase\expandafter{\romannumeral2}, 1979, pp. 249-268.

\bibitem[10]{Moller} H. M\"{o}ller,
\"{U}ber die Koeffizienten des n-ten Kreisteilungspolynoms, Math. Z.
119 (1971) 33--40.

\bibitem[11]{Zhao} Jia Zhao, Xianke Zhang, On the
coefficients of the cyclotomic polynomials of order three,
arXiv:0910.1982, preprint.

\end{thebibliography}
\end{document}